\documentclass[11pt]{article}
\usepackage{amsmath,amsthm,amsfonts,amssymb,amscd, amsxtra}
\usepackage{color}
\usepackage{setspace} 
\usepackage{graphicx}
\usepackage{subfig}
\newtheorem{theorem}{Theorem}
\newtheorem{lemma}[theorem]{Lemma}
\newtheorem{definition}{Definition}

\newtheorem{remark}{Remark}

\newtheorem{algorithm}{Algorithm}

\begin{document}
\title{Inexact Newton method with feasible inexact projections for solving constrained smooth and  nonsmooth equations}

\author{
    F.R. de Oliveira
    \thanks{Instituto de Matem\'atica e Estat\'istica, Universidade Federal de Goi\'as, Campus II- Caixa
    Postal 131, CEP 74001-970, Goi\^ania-GO, Brazil. (E-mails: {\tt
       rodriguesfabiana@discente.ufg.br} (\textbf{Corresponding Author}) and {\tt orizon@ufg.br}). The authors were supported in part by CNPq grants 408151/2016-1 and 302473/2017-3,  FAPEG/GO, and CAPES.}
  \and O. P. Ferreira  \footnotemark[1]
}
 \maketitle
\begin{abstract}
In this paper, we propose a new method that combines the inexact Newton method with a procedure to obtain a feasible inexact projection for solving constrained smooth and nonsmooth equations. The local convergence theorems are established under the assumption of smoothness or semismoothness of the function that defines the equation and its regularity at the solution. In particular, we show that a sequence generated by the method converges to a solution with linear, superlinear, or quadratic rate, under suitable conditions. Moreover, some numerical experiments are reported to illustrate the practical behavior of the proposed method.
\end{abstract}



\maketitle

\section{Introduction} \label{intro}
Unconstrained nonsmooth equations  are of great interest in mathematical programming, considering that it  arises  from the reformulation of important problems, such as the nonlinear complementarity problem, the variational inequality problem, and the nonlinear programming problem. See \cite{Bertsekas1999,Facchinei2003,SteveJong1992,NocedalWright2000,PangQi19993,SolodovSvaiter1999} and references therein. Owing to the large number of applications where these equations appear, numerical techniques have been proposed to solve them. For instance, in \cite{JoseMariLiqun1995} a version of inexact Newton method was presented for solving the unconstrained nonsmooth equation
\begin{equation}\label{eq:SemsEq}
f(x)=0,
\end{equation}
where $f:{\Omega}\to {\mathbb R}^n$ is a locally Lipschitz continuous function and $\Omega\subseteq {\mathbb R}^n$ is an open set, which have  the following formal  formulation: For the current iterate $x_k \in \mathbb{R}^n$, the next iterate is any point $x_{k+1} \in \mathbb{R}^n$   satisfying the relative residual  error criteria
\begin{equation}\label{eq:CondResidual}
\|f(x_k) + V_k (x_{k+1}-x_k)\| \leq \eta_k \|f(x_k)\|,
\end{equation}
where $\eta_k \in [0,1)$ is  the relative residual  error tolerance and  $V_k$ is an element of the Clarke generalized Jacobian of $f$ at $x_k$ (for the definition of the Clarke generalized Jacobian, see  \cite{Clarke1990}). More versions of  inexact Newton-type methods for solving \eqref{eq:SemsEq}  include, but are not limited to, those in  \cite{BirginKrejicMartines2003,BonettiniTinti2007,Facchinei2003,PuTian2002,Smieta2007,MarekJS2012,Smieta2013}.

Our aim in this paper is to study  the  inexact  Newton method \eqref{eq:CondResidual} with feasible inexact projections ({\it inexact  Newton-InexP method})  for solving smooth and nonsmooth equations subject to a set of constraints, i.e., to solve the following constrained equation: Find $x \in \mathbb{R}^n$ such that
\begin{equation} \label{eq:NonEq}
x \in C, \qquad f(x) =  0,
\end{equation}
where $f:{\Omega}\to {\mathbb R}^n$ is a locally Lipschitz continuous function,  $\Omega\subseteq {\mathbb R}^n$ is an open set and  $C \subset \Omega$ is a nonempty closed convex set. If $f$ is a continuously differentiable function, then  \eqref{eq:NonEq} reduces to a constrained smooth equation, which can be easily found in the literature: see, for instance, \cite{BellaviMaria2004,Bellavia2012,Jones:2000,mariniquasi2018,morini2016}. The problem of solving \eqref{eq:NonEq} has been addressed in several studies, and several similar methods and/or variants of \eqref{eq:CondResidual}  have been proposed for solving it. See, for example, the exact/inexact Newton-like methods in  \cite{GoncalvesOliveira2017,MaxJefferson2017,mariniquasi2018}, projected Levenberg--Marquardt-type methods  in \cite{BehlingHaeserRamosSchonefeld2017,BehlingFischerHerrichIusemYe2014}, and  trust-region methods in \cite{BellaviMaria2004,Bellavia2012}. In particular,  the method proposed in \cite{BehlingFischerHerrichIusemYe2014} combines a  Levenberg--Marquardt-type method with an inexact projection, which also accepts an infeasible inexact projection.

In the present  paper, we  propose a scheme for solving  \eqref{eq:NonEq}, which we call the {\it inexact  Newton-InexP method}, that also uses the concept of inexact projection. However, inexact projections used in this scheme  are always feasible. In essence, the proposed method combines the inexact Newton method   with  a procedure to obtain a feasible inexact  projection onto  $C$ and thus  ensure the viability of the iterates. The concept of a feasible inexact projection used was introduced in \cite{deOliveiraFerreiraSilva2018}, which  also accepts an exact projection that  can be adopted  when it is easily obtained. For instance, the exact projections onto a box constraint or Lorentz cone are very easily   obtained; see \cite[p. 520]{NocedalWright2000} and \cite[Proposition 3.3]{FukushimaTseng2002}, respectively. It is worth mentioning that a feasible inexact projection can be computed by any method  that minimizes a quadratic function on closed convex set efficiently by introducing a suitable error criteria. For instance, if the set $C$ is polyhedral, then some iterations of an interior point method or  active set method can be performed  for finding a feasible inexact projection, see \cite{NicholasPhilippe2002,NocedalWright2000,Robert1996}. When $ C $ is a simple convex compact set,  a similar scheme was also adopted  in \cite{GoncalvesOliveira2017,MaxJefferson2017,lan2016}, which used the conditional gradient  method to find a feasible inexact projection. An issue to consider is the inexact solution in \eqref{eq:CondResidual}, which  has  an advantage over the exact solution,  see \cite{DemboRonTrond1983}. This advantage appears more explicitly in practical implementations of the method,  because finding the exact solution of  linear approximations of \eqref{eq:SemsEq} can be computationally very expensive for large-scale problems. Thus,  in the present  paper,  we consider that  from  the  current  iterate,  the next iterate is any point in $C$ satisfying  the  relative residual error  criteria \eqref{eq:CondResidual}. Finally, when $C = {\mathbb R}^n$, the inexact Newton-InexP method becomes the classical inexact Newton method applied to nonsmooth  equations, see  \cite{JoseMariLiqun1995}.

From the theoretical viewpoint, i.e.,  in the convergence analysis presented, to guarantee local efficiency of the proposed method,  we assume appropriate assumptions, such as  regularity and semismoothness. Under the regularity assumption, we ensure that locally a sequence generated by the method is well-defined. The  semismoothness assumption is of particular interest owing to the key role it plays in the convergence of our method; in particular, this property is essential for fast local convergence. To illustrate the robustness and efficiency of the new method, we present some preliminary numerical experiments of the proposed method for solving the constrained absolute value equation (CAVE). We also compare the performance of the proposed  method with the inexact Newton method with feasible exact projections.

This paper is organized as follows. In Section~\ref{sec:int.1}, we present the notations and some technical definitions that are used throughout the paper. In Section~\ref{Sec:LocalAnalysisInexact}, we describe  the  inexact Newton-InexP method and we study its local convergence properties. In Section~\ref{sec:specas}, we present two applications of the main convergence theorems. Numerical experiments are presented in
Section~\ref{NunEx}. We finish the paper with some remarks in Section~\ref{Sec:Conclusions}.

\section{Notations and definitions} \label{sec:int.1}
In this section, we present  some notations, definitions, and results used throughout the paper. For further details, see, for example, \cite{Clarke1990,DontchevRockafellar2010Book,Facchinei2003}.

Let $ B_{\delta}(x) := \{ y \in {\mathbb R}^n: ~\|x-y\|<\delta \}$ be   the {\it open ball} of radius $\delta > 0$ centered at $x$.
The {\it norm} of a linear mapping $A:{\mathbb R}^n \to {\mathbb R}^n$ is denoted by $\|A\|:=\sup \,\{\|A x\|:~\|x\| \leq 1 \}$.  In the following, we define the concept of a locally Lipschitz continuous function, which plays an important role in our study.
\begin{definition}\label{def:LipsCont}
Let $\Omega \subset \mathbb{R}^n$.  A function $f: \Omega  \to \mathbb{R}^m$  is said to be locally Lipschitz continuous if for each $x\in \Omega$, there exist constants $\alpha> 0$ and $\delta> 0$ such that  $\|f(y) - f(z)\| \leq \alpha \|y - z\|$, for all $y, z \in B_{\delta}(x)$.
\end{definition}
\begin{remark}
According to the Rademacher theorem, see \cite[Theorem~2, p.~81]{Evans1992}, locally Lipschitz continuous functions are differentiable almost everywhere.
\end{remark}

Now, we define the Clarke generalized Jacobian of a function, which has appeared in \cite{Clarke1990}. This Jacobian requires only local  Lipschitz continuity of the function $f$  and its well-definedness is ensured by the Rademacher theorem.

\begin{definition}\label{def:JacClarke}
The Clarke generalized Jacobian of a locally Lipschitz continuous function $f$ at $x$ is a set-valued mapping $\partial f: \mathbb{R}^n \rightrightarrows \mathbb{R}^m$ defined as
$$
\partial f(x) := \mbox{co}\left\{H \in \mathbb{R}^{m\times n}:~ \exists \, \{x_k\} \subset \mathcal{D}_f, \lim_{k \to +\infty} x_k = x,\, H = \lim_{k \to + \infty}f'(x_k)\right\},
$$
where ``\mbox{co}'' represents the convex hull, $\mathbb{R}^{m\times n}$ is the set consisting of all $m\times n$ matrices, and  $\mathcal{D}_f$ denotes  the set of points at which  $f$ is differentiable.
\end{definition}

\begin{remark}
It is worth mentioning that if $f$ is continuously differentiable at $x$, then $\partial f(x) = \{f'(x)\}$. Otherwise, $\partial f(x)$ could contain other elements different from $f'(x)$, even if $f$ is differentiable at $x$, see \cite[Example 2.2.3]{Clarke1990}. Furthermore, the Clarke generalized Jacobian is a subset of $\mathbb{R}^{m\times n}$ that is nonempty, convex, and compact in the usual sense. We also remind that the set-valued mapping $\partial f$  is closed and upper semicontinuous, see \cite[Proposition 2.6.2, p.~70]{Clarke1990}.
\end{remark}

\section{Inexact Newton-InexP method} \label{Sec:LocalAnalysisInexact}
In this section, we present the inexact Newton-InexP method for solving the problem \eqref{eq:NonEq}, where the function $f$ is locally Lipschitz  continuous.  Basically, the inexact Newton-InexP method combines the  inexact version of Newton method (see, for instance, \cite{Facchinei2003,JoseMariLiqun1995}) with a procedure to obtain a feasible inexact projection.   We begin  by presenting the concept of a feasible inexact projection.

\begin{definition} \label{def:InexactProj}
Let $C\subset {\mathbb R}^n$ be a closed convex set,  $x\in C$, and $\theta\geq 0$.  The {\it feasible inexact projection mapping}  relative to  $x$ with  error tolerance $\theta$, denoted by  $P_C(\cdot,x,\theta):  {\mathbb R}^n \rightrightarrows C$, is the set-valued mapping defined as follows
$$
P_C(y,x,\theta):=\left\{ w\in C: ~  \left\langle y-w, z-w \right\rangle \leq \theta \|y-x\|^2, \quad \forall~z\in C  \right\}.
$$
Each  point $w\in P_C(y,x,\theta) $ is  called a {\it feasible inexact projection of $y$ onto $C$ with respect to  $x$  and  error tolerance $\theta$}.
\end{definition}

As $C\subset {\mathbb R}^n$ is  a closed convex set,  \cite[Proposition~2.1.3, p. 201]{Bertsekas1999} implies that  for each $y\in {\mathbb R}^n$  and $x\in C$  we have   $\{P_{C}(y)\}=P_C(y,x,0) \subset   P_C(y,x,\theta)$, where  $P_{C}$ denotes the exact projection mapping. Hence, $P_C(y,x,\theta)\neq \varnothing$,  for all  $y\in {\mathbb R}^n$ and $x\in C$, and consequently the mapping $P_C(\cdot,x,\theta):  {\mathbb R}^n \rightrightarrows C$ is well-defined.
\begin{remark}
It is worth mentioning that the conditional gradient procedure (\textit{CondG procedure}; see, for instance, \cite{FrankWolfe1956,lan2016}), which is  based on the conditional gradient method, is an  example of  the procedure for obtaining feasible inexact projections onto  special compact  sets $C$. For a general overview of this method, see \cite{Bertsekas1999}. The concept of inexact projection has been considered before;  see, for example, \cite{BehlingFischerHerrichIusemYe2014}.  We remark that, in general, those inexact  projections are infeasible and, thus, different from the above concept.
\end{remark}

The next result plays an important role in the remainder of this paper. It  presents an important  property of the feasible inexact projection.  It is worth noting that   it  is a generalization  of \cite[Lemma 4]{MaxJefferson2017}  for a general  feasible inexact projection.
\begin{lemma}  \label{pr:condi}
Let  $y, {\tilde y}\in {\mathbb R}^n$, $x, {\tilde x} \in C$, and $\theta \geq 0$. Then, for any  $w \in  P_C(y, x, \theta) $, we have
$$
\left\|w -  P_C({\tilde y}, {\tilde x}, 0)\right\| \leq  \|y-{\tilde y}\|+ \sqrt{2\theta}\|y-x\|.
$$
\end{lemma}
\begin{proof}
To simplify the notation, we set ${\tilde w} = P_C({\tilde y}, {\tilde x}, 0)$, and take $w \in P_C(y, x, \theta)$. First, note that
$$\|y - {\tilde y}\|^2 = \|(y-w) - ({\tilde y} - {\tilde w})\|^2 + \|w - {\tilde w}\|^2 + 2 \langle (y -{\tilde y}) - (w- {\tilde w}), w - {\tilde w} \rangle,$$ which implies that
$$
 \|w - {\tilde w}\|^2 \leq \|y - {\tilde y}\|^2 + 2 \langle y - w, {\tilde w}-w \rangle + 2 \langle {\tilde y} - {\tilde w}, w - {\tilde w} \rangle.
 $$
Because ${\tilde w} = P_C({\tilde y}, {\tilde x}, 0)$ and $w \in P_C(y, x, \theta)$, by using Definition~\ref{def:InexactProj} and the fact that ${\tilde w}, w\in C$, we can conclude that $\langle y-w,\, {\tilde w} - w\rangle \leq \theta \|y - x\|^2$ and $ \langle {\tilde y}-{\tilde w},\, w - {\tilde w}\rangle \leq 0$. Thus, the combination of these three previous inequalities yields $\|w - {\tilde w}\|^2 \leq \|y - {\tilde y}\|^2 + 2 \theta \|y - x\|^2$, and then $\|w - {\tilde w}\| \leq \|y - {\tilde y}\|+ \sqrt{2 \theta} \|y - x\|$, giving the desired inequality.
\end{proof}

In this section, we assume that  $f: \Omega \to {\mathbb R}^n$ is a locally Lipschitz continuous function.  Now, we formally describe the algorithm for solving the problem~\eqref{eq:NonEq}.
\noindent
\\
\hrule
\begin{algorithm}\label{Alg:INP}
{\bf Inexact Newton-InexP method\\}
\hrule
\begin{description}
\item[ \textbf{Step 0.}] Let $ {\theta}>0$,  $ {\eta}>0$,  $x_0\in C$,  $\{\theta_k\}\subset [0, {\theta})$, and $\{\eta_k\}\subset [0, {\eta})$ be given and set $k=0$.
\item[ \textbf{Step 1.}] If $f(x_k) = 0$, then {\bf stop}; otherwise, choose an element $V_k \in \partial f(x_k)$ and compute $y_k\in{\mathbb R}^n$ such that
\begin{equation}\label{eq:semsnew}
\|f(x_k)+V_k(y_k-x_k)\|\leq \eta_k \|f(x_k)\|.
\end{equation}
\item[ \textbf{Step 2.}]  If $y_k \in C$, set $x_{k+1} = y_k$; otherwise, use a procedure to obtain $P_C(y_k, x_k, \theta_k)\in C$ a feasible inexact projection of $y_k$ onto $C$ relative to  $x_k$ with  relative error tolerance $\theta_k$; and set
\begin{equation*} \label{eq:semscond}
x_{k+1} \in P_C\left(y_{k},x_k,\theta_k\right).
\end{equation*}
\item[ \textbf{Step 3.}]  Set $k\gets k+1$, and go to \textbf{Step~1}.
\end{description}
\hrule
\end{algorithm}
\noindent

In the following, we describe the main features of the inexact Newton-InexP method.

\begin{remark}
In inexact Newton-InexP method, we first check whether the current iterate $x_k$ is a solution of the problem \eqref{eq:NonEq}; otherwise, we compute $y_k$ satisfying the relative residual error criteria~\eqref{eq:semsnew}. The forcing sequence $\{\eta_k\}$ is used to control the level of accuracy. In particular, as we show, the specific choice of this sequence is essential to establish the local convergence of the inexact Newton-InexP method. It is worth pointing out that if $\eta_k = 0$ for all $k = 0,1, \ldots$, i.e., the exact version of the Newton-InexP method, then $y_k$  is obtained by solving the system $f(x_k)+V_k(y_k-x_k) = 0$. Note that to ensure the well-definedness of $y_k$ the Clarke generalized Jacobian must be nonempty, see \cite[Proposition~2.6.2, p. 70]{Clarke1990},  and all $V_k\in \partial f(x_k)$ must be nonsingular, for any $k = 0, 1, \ldots$. As the point $y_k$  can be infeasible for the set of constraints $C$, the inexact Newton-InexP method uses a procedure to obtain a feasible inexact projection, and consequently the new iterate $x_{k+1}$ belongs to $C$. The choice of the tolerance $\theta_k$ is also important in obtaining the local convergence of the inexact Newton-InexP method. Finally, we remark that  if $f$ is continuously differentiable, $\eta_k = 0$ for all $k = 0,1, \ldots$, and the procedure to obtain $P_C(y_k, x_k, \theta_k)$ is the CondG procedure, then  our method is equivalent to the method proposed in \cite{MaxJefferson2017}. On the other hand, if $C = {\mathbb R}^n$ and $\eta_k = \theta_k = 0$ for all $k = 0,1, \ldots$,  our method reduces to Newton method proposed in \cite{Qi1993}.
\end{remark}

In the following, we state and prove our first local convergence result for a sequence generated by the inexact Newton-InexP method.
\begin{theorem}\label{th:conv}
Let   $\Omega \subseteq {\mathbb R}^n$ be an open set, $C\subset\Omega$ be  a closed convex set, and  $f: \Omega \to {\mathbb R}^n$ be a locally Lipschitz continuous function.  Suppose   that ${\bar x}\in C$ and   $f({\bar x}) = 0$. Let  $\Gamma>0$ and $0<r \leq {\bar r}:=\sup\left\{t\in {\mathbb R}:~ B_t({\bar x})\subset \Omega\right\}$ such that
\begin{equation} \label{eq:gamma}
\|f(x) - f({\bar x})\| \leq \Gamma \|x - {\bar x}\|, \qquad \forall~x\in  B_{r}({\bar x}).
\end{equation}
Assume that each $ { V_{\bar x}}\in \partial f({\bar x})$ is nonsingular and  let   $\lambda_{\bar x}\geq \max\{\|{V_{\bar x}^{-1}}\|: ~ { V_{\bar x}}\in \partial f({\bar x})\}$. Moreover,   there exist  $\epsilon>0$ and $0<\delta\leq \min\{ r, 1\}$ such that for  all $x\in   B_{\delta}({\bar x})$,  $V_{x}\in \partial f(x)$ is nonsingular and  the following hold:
\begin{eqnarray}
                                                                                                \displaystyle \|V_{x}^{-1}\| &\leq&  \frac{\lambda_{{\bar x}}}{1 -\epsilon \lambda_{{\bar x}}},   \label{eq:fcA1}    \\
     \displaystyle \|f({\bar x})-f(x)-V_{x}({\bar x}-x)\|&\leq& \epsilon \|x - {\bar x}\|^{1+\mu},  \qquad 0\leq \mu \leq 1.\label{eq:scA1}
\end{eqnarray}
Let   $0< {\theta}<1/2$. Furthermore,  assume that   ${\eta}> 0$ and  $\epsilon>0$  satisfy the following conditions
\begin{equation} \label{eq:DeltaEtaT1}
{\eta}<\frac{1-\sqrt{2{\theta}}}{{\lambda_{{\bar x}}}\Gamma\left(1+\sqrt{2{\theta}}\right)}, \qquad
\epsilon< \displaystyle \frac{1}{2\lambda_{{\bar x}}} \left[\left(1-\sqrt{2{\theta}}\right)-{\eta} {\lambda_{{\bar x}}}\Gamma\left(1+\sqrt{2{\theta}}\right)\right].
 \end{equation}
Then,   every sequence $\{x_k\}$  generated by Algorithm~\ref{Alg:INP} starting in $x_0 \in C \cap B_\delta({\bar x})\backslash \{{\bar x}\}$, with $0\leq\eta_k <{\eta}$ and $0\leq \theta_k< {\theta}$, for all $k=0, 1, \ldots$,  belongs to  $ B_\delta({\bar x})\cap C$,  satisfies
\begin{equation} \label{eq:BoundIneT1}
\|x_{k+1} - {\bar x}\|  \leq \left[\frac{\lambda_{\bar x}[\eta_k\Gamma + \epsilon \|x_k - {\bar x}\|^{\mu}]}{1- \epsilon{\lambda_{{\bar x}}} }\left(1+\sqrt{2\theta_k}\right) +  \sqrt{2\theta_k}\right] \|x_k-{\bar x} \|,
\end{equation}
for all $k = 0, 1, \ldots$, and converges $Q$-linearly to ${\bar x}$. As a consequence, if $\lim_{k \to +\infty} \theta_k = 0$ and $\lim_{k \to +\infty} \eta_k = 0$, then $\{x_k\}$ converges $Q$-superlinearly    to ${\bar x}$. Furthermore,    letting  $\eta_k <  \min \{ {\eta}\|f(x_k)\|^{\mu}, \eta\}$  and $\theta_k < \min \{{\theta}\|f(x_k)\|^{2\mu}, \theta\}$,  the convergence of $\{x_k\}$ to ${\bar x}$ is of the order of $1 + \mu$.
\end{theorem}
\begin{proof}
We show by induction on $k$ that if $x_0 \in C \cap B_\delta({\bar x})\backslash \{{\bar x}\}$,  then every  sequence $\{x_k\}$ generated by  Algorithm~\ref{Alg:INP} belongs to $B_\delta({\bar x})\cap C$ and  satisfies  \eqref{eq:BoundIneT1}. Indeed, set  $k = 0$,  and  take $\theta_0\geq 0$, $\eta_0\geq 0$, $x_{0} \in  C \cap B_{\delta}({\bar x})\backslash \{{\bar x}\}$ and $V_{x_0} \in \partial f(x_0)$. Owing to $\|x_0 - {\bar x}\|< \delta$, we have  $V_{x_0}$ is nonsingular and  then $y_0$  in \eqref{eq:semsnew} is well-defined for $k=0$. As $f({\bar x}) = 0$, we have
$$
y_0 - {\bar x} = V_{x_0}^{-1}\left(\left[f(x_0)+V_{x_0}(y_0 -x_0)\right] + \left[f({\bar x})-f(x_0) - V_{x_0}( {\bar x}-x_0)\right]\right).
$$
Taking the norm on both sides of the last inequality and using the triangular inequality,  we conclude that
$$
\|y_0 - {\bar x}\|\leq  \left\|V_{x_0}^{-1}\right\|\Big[\big\| f(x_0)+ V_{x_0}(y_0 -x_0) \big\| + \big\|f({\bar x}) - f(x_0) -V_{x_0}( {\bar x}-x_0)\big\| \Big].
$$
Thus, using \eqref{eq:semsnew}  with $k=0$ and the assumptions \eqref{eq:fcA1} and \eqref{eq:scA1} with $x=x_0$, we obtain  that
$$
\|y_0 - {\bar x}\| \leq \frac{\lambda_{\bar x}}{1 - \epsilon \lambda_{\bar x}}\left[\eta_0 \|f(x_0)\| + \epsilon\|x_0 - {\bar x}\|^{1+\mu} \right].
$$
As $f({\bar x}) = 0$,  from  \eqref{eq:gamma} we have   $\|f(x_0)\|  \leq \Gamma\|x_0 - {\bar x}\|$. Hence, the last inequality  becomes
\begin{equation}\label{eq:lim21}
\|y_0 - {\bar x}\| \leq \frac{\lambda_{\bar x}[\eta_0 \Gamma + \epsilon \|x_0 - {\bar x}\|^{\mu}]}{1 - {\epsilon\lambda_{\bar x}}}\|x_0 - {\bar x}\|.
\end{equation}
Taking any  $x_1 \in P_C(y_0, x_0, \theta_0)$ and applying Lemma~\ref{pr:condi} with $y = y_0$, $x = x_0$, $\theta = \theta_0$, $\tilde y={\bar x}$ and $\tilde x={\bar x}$, we have
$$
\left\|x_1 -{\bar x} \right\| \leq  \left\|y_0-{\bar x}\right\| + \sqrt{2\theta_0} \left\|y_0 - x_0\right\|\leq \left\|y_0-{\bar x}\right\|  \left(1 + \sqrt{2\theta_0} \right) + \sqrt{2\theta_0} \left\|x_0-{\bar x} \right\|.
$$
Combining  \eqref{eq:lim21} with  the last inequality, we obtain that
$$
\left\|x_1-{\bar x} \right\|\leq \frac{\lambda_{\bar x}[\eta_0\Gamma + \epsilon \|x_0 - {\bar x}\|^{\mu}]}{1 - {\epsilon\lambda_{\bar x}}}\left\|x_0-{\bar x} \right\| \left(1 + \sqrt{2\theta_0}\right) + \sqrt{2\theta_0} \left\|x_0-{\bar x} \right\|,
$$
which is equivalent to \eqref{eq:BoundIneT1} with $k=0$. As $\delta \leq 1$,  $\eta_0 < \eta$, and $\theta_0 < \theta$,  by using \eqref{eq:DeltaEtaT1}, we obtain
$$
\frac{\lambda_{\bar x}[\eta_0\Gamma + \epsilon\|x_0 - {\bar x}\|^{\mu}]}{1- \epsilon{\lambda_{{\bar x}}} }\left(1+\sqrt{2\theta_0}\right) +  \sqrt{2\theta_0}< \frac{\lambda_{\bar x}[\eta\Gamma + \epsilon]}{1- \epsilon{\lambda_{{\bar x}}} }\left(1+\sqrt{2\theta}\right) +  \sqrt{2\theta} < 1.
$$
Thus, because $x_0\in B_{\delta}({\bar x})$, we obtain  from  \eqref{eq:BoundIneT1} with $k=0$  that  $\left\|x_1 -{\bar x}  \right\| <  \|x_0 -{\bar x} \|<\delta$. As $P_C(y_0, x_0, \theta_0)$ belongs to $C$ and $ x_1\in P_C\left(y_0, x_0, \theta_0\right)$,  we conclude that  $ x_1$   belongs to  $B_\delta({\bar x})\cap C$, which completes the induction step for $k=0$.   The general induction step  is completely analogous.  Therefore,  every  sequence $\{x_k\}$ generated by Algorithm~\ref{Alg:INP} is contained in $B_\delta({\bar x})\cap C$ and  satisfies  \eqref{eq:BoundIneT1}.  We proceed to prove that the sequence $\{x_k\}$ converges to ${\bar x}$. As $\delta \leq1$, $0\leq\theta_k<{\theta}$, and $0 \leq\eta_k <{\eta}$ for all $k=0, 1, \ldots$, it follows from \eqref{eq:BoundIneT1} and \eqref{eq:DeltaEtaT1} that
\begin{equation} \label{eq:Convproof12}
\|x_{k+1}- {\bar x}\| < \left[\frac{\lambda_{\bar x}[\eta\Gamma + \epsilon ]}{1- \epsilon{\lambda_{{\bar x}}} }\left(1+\sqrt{2\theta}\right) +  \sqrt{2\theta}\right]\|x_k - {\bar x}\|< \|x_k - {\bar x}\|,
\end{equation}
for all $k = 0, 1, \ldots$. This implies that the sequence $\{\|x_k - {\bar x} \|\}$ converges.  Let us say that ${\bar t}:=\lim_{k\to +\infty}\|x_k - {\bar x}\|\leq \delta$.   Thus, taking  the limit   in  \eqref{eq:Convproof12}   as  $k$ goes to $+ \infty$,  we have
$$
{\bar t} \leq \left[\frac{\lambda_{\bar x}[\eta\Gamma + \epsilon ]}{1- \epsilon{\lambda_{{\bar x}}} }\left(1+\sqrt{2\theta}\right) +  \sqrt{2\theta}\right]{\bar t},
$$
If ${\bar t}\neq 0$, then  \eqref{eq:DeltaEtaT1} implies that  $ {\bar t} < {\bar t}$, which is absurd.  Hence,  ${\bar t} = 0$ and  $\{x_k\}$ converges  $Q$-linearly to ${\bar x}$.  Now,  we assume that $\lim_{k \to +\infty} \theta_k = 0$ and $\lim_{k \to +\infty} \eta_k = 0$. Thus,    for $\mu=0$,  it  follows  from \eqref{eq:BoundIneT1}  that
$$
\lim_{k\to +\infty} \frac{\|x_{k+1}- {\bar x}\|}{\|x_{k}- {\bar x}\|}=\frac{\epsilon\lambda_{\bar x}}{1-\epsilon  \lambda_{\bar x}},
$$
and,   by  taking into account that  $\epsilon>0$  is any number  satisfying \eqref{eq:DeltaEtaT1}, we conclude that $\{x_k\}$ converges $Q$-superlinearly    to ${\bar x}$. For  $0<\mu \leq 1$ it follows straight from \eqref{eq:BoundIneT1} that   $\{x_k\}$ converges $Q$-superlinearly    to ${\bar x}$.  Finally, we assume that  $\eta_k <  \min \{ {\eta}\|f(x_k)\|^{\mu}, \eta\}$ and $\theta_k < \min \{{\theta}\|f(x_k)\|^{2\mu}, \theta\}$.   Considering that  $\{x_k\}$   belongs to  $ B_\delta({\bar x})$, $f({\bar x}) = 0$, and $\delta \leq r$, it follows from  \eqref{eq:gamma} that $\|f(x_k)\|  \leq \Gamma\|x_k - {\bar x}\|$ for all $ k = 0, 1, \ldots$. Hence, $\eta_k < \eta \Gamma^{\mu}\|x_k - {\bar x}\|^{\mu}$ and  $\theta_k < \theta \Gamma^{2\mu}\|x_k - {\bar x}\|^{2\mu}$ for all $ k = 0, 1, \ldots$. Then, \eqref{eq:BoundIneT1} implies
$$
\|x_{k+1} - {\bar x}\|  < \left[\frac{\lambda_{\bar x}[\eta\Gamma^{1+\mu} + \epsilon]}{1- \epsilon{\lambda_{{\bar x}}} }\left(1+\Gamma^{\mu}\sqrt{2\theta}\|x_k - {\bar x}\|^{\mu}\right) +  \Gamma^{\mu}\sqrt{2\theta} \right] \|x_k-{\bar x} \|^{1+\mu},
$$
for all $k = 0, 1, \ldots$. Therefore, $\{x_k\}$ converges to ${\bar x}$ with  order $1 + \mu$, and the proof of the theorem is complete.
\end{proof}

In the following remark, we present a particular case of Theorem~\ref{th:conv}, i.e., when  the projection and Newton method are exact.
\begin{remark}
Note that the mapping   $(0, 1/2)\ni \theta \mapsto (1-\sqrt{2{\theta}})/(1+\sqrt{2{\theta}})$ is decreasing. Thus, from the first inequality in \eqref{eq:DeltaEtaT1}, we conclude  that if $\theta$ approaches the upper bound $1/2$, then $\eta$ approaches the lower bound $0$. Therefore,    in Algorithm~\ref{Alg:INP}, the most inexact is the projection, the least inexact has to be the Newton direction. Moreover, it follows from \eqref{eq:BoundIneT1} that  if $\theta_k\equiv 0$ and $\eta_k\equiv 0$,   then  the convergence rate of $\{x_k\}$ has order $1 + \mu $ as follows
$$
\| x_{k+1} - {\bar x}\| \leq \frac{\lambda_{\bar x} \epsilon }{1-  \epsilon\lambda_{\bar x}} \|x_k - {\bar x}\|^{1 + \mu}, \qquad k = 0, 1, \ldots.
$$
Hence, $\epsilon$ in \eqref{eq:DeltaEtaT1} is related to the bound for  the convergence rate.
\end{remark}

Next, we state and prove our second local convergence result for a sequence generated by the inexact Newton-InexP method. In this case, we assume that $f: \Omega  \to \mathbb{R}^n$  is a continuously differentiable function.
\begin{theorem}\label{th:mainInex}
Let $ \Omega \subseteq  \mathbb{R}^n$  be an open set,  $C \subset \Omega$  be a closed convex set, and  $f: \Omega  \to \mathbb{R}^n$  be a  continuously differentiable function.  Suppose that  ${\bar x}\in C $ and   $f({\bar x}) = 0$. Let     $\Gamma>0$ and $0<r \leq {\bar r}:=\sup\left\{t\in {\mathbb R}:~ B_t({\bar x})\subset \Omega\right\}$ such that
\begin{equation} \label{eq:gamma1}
\|f(x) - f({\bar x})\| \leq \Gamma \|x - {\bar x}\|, \qquad \forall~x\in  B_{ r}({\bar x}).
\end{equation}
Assume that $  f'({\bar x})$ is nonsingular and  there exist $0< \mu \leq 1$, $L>0$, and $0<{\hat \delta}\leq r$ such that for all $x \in B_{\hat \delta}({\bar x})$,     $f'({x})$ is nonsingular and the following hold:
\begin{eqnarray}
\displaystyle \|f'(x)^{-1}\| &\leq&  \frac{ \|f'({\bar x})^{-1}\|}{1 -  L\|f'({\bar x})^{-1}\|\|x - {\bar x}\|^\mu},   \label{eq:fsecA1}    \\
\displaystyle \|f({\bar x})-f(x)-f'(x)({\bar x}-x)\|&\leq& \frac{\mu L}{1+\mu} \|x - {\bar x}\|^{1+\mu}.  \label{eq:ssecA1}
\end{eqnarray}
Furthermore,  let   $0<{\theta}<1/2$,   ${\eta}> 0$, and ${ \delta}>0$  satisfying the following  conditions
\begin{eqnarray}
 \displaystyle {\eta} &< &\frac{1-\sqrt{2{\theta}}}{\Gamma \|f'({\bar x})^{-1}\|\left(1+\sqrt{2{\theta}}\right)},   \label{eq:EtaT2}    \\
\displaystyle {\delta} &<& \min \left\{   {\hat \delta}, \left[ \frac{(1+\mu)\left[\left(1-\sqrt{2{ \theta}}\right)-{\eta} \Gamma \|f'({\bar x})^{-1}\|\left(1+\sqrt{2{\theta}}\right)\right]}{L\left[1 + 2 \mu - \sqrt{2\theta}\right] \|f'({\bar x})^{-1}\| }\right]^{1/\mu}\right\}. \label{eq:DeltaT2}
\end{eqnarray}
Then, every sequence $\{x_k\}$  generated by Algorithm~\ref{Alg:INP} starting in  $x_0 \in C \cap B_\delta({\bar x})\backslash \{{\bar x}\}$, with $0\leq\eta_k < {\eta}$ and $0\leq \theta_k< {\theta}$ for all $k=0, 1, \ldots$,   is  contained in $ B_\delta({\bar x})\cap C$, satisfies
\begin{equation} \label{eq:BoundIneT2}
\|x_{k+1} - {\bar x}\|  \leq  \left[\frac{ \|f'({\bar x})^{-1}\|[\eta_k \Gamma(1 + \mu) + \mu L\|x_k-{\bar x} \|^{\mu}] }{(1+\mu)[1- L{ \|f'({\bar x})^{-1}\|}\|x_k-{\bar x} \|^\mu]}\left(1+\sqrt{2\theta_k}\right) + \sqrt{2\theta_k}\right] \|x_k-{\bar x} \|,
\end{equation}
for all $k = 0, 1, \ldots$, and converges $Q$-linearly to ${\bar x}$.  As a consequence, if $\lim_{k \to +\infty} \theta_k = 0$ and $\lim_{k \to +\infty} \eta_k = 0$, then $\{x_k\}$ converges $Q$-superlinearly    to ${\bar x}$. Furthermore,    letting  $\eta_k <  \min \{ {\eta}\|f(x_k)\|^{\mu}, \eta\}$  and $\theta_k < \min \{{\theta}\|f(x_k)\|^{2\mu}, \theta\}$,  the convergence of $\{x_k\}$ to ${\bar x}$ is of the order of $1 + \mu$.
\end{theorem}
\begin{proof}
First, note that as $f$ is continuously differentiable  at $x$, we have $\partial f(x) = \{f'(x)\}$. We show by induction on $k$ that if $x_0 \in C \cap B_\delta({\bar x})\backslash \{{\bar x}\}$,  then every  sequence $\{x_k\}$ generated by Algorithm~\ref{Alg:INP} is contained in $B_\delta({\bar x})\cap C$ and  satisfies  \eqref{eq:BoundIneT2}. To this end, take $\theta_0\geq 0$, $\eta_0\geq 0$, $x_{0} \in  C \cap B_{\delta}({\bar x})\backslash \{{\bar x}\}$, and set $k = 0$. Owing to $\|x_0 - {\bar x}\|< \delta$,  we obtain that  $f'(x_0)$ is nonsingular. Consequently,   \eqref{eq:semsnew} with  $k=0$ and  $V_{x_0} = f'(x_0)$, implies that   $y_0$  is well-defined. Because $f({\bar x}) = 0$, after some algebraic manipulations,  we have
$$
y_0 - {\bar x} = f'(x_0)^{-1}\left(\left[f(x_0)+ f'(x_0)(y_0 -x_0)\right] + \left[f({\bar x})-f(x_0) - f'(x_0)( {\bar x}-x_0)\right]\right).
$$
Taking the norm on both sides of the last inequality and using triangular inequality,  we conclude that
$$
\|y_0 - {\bar x}\|\leq \left\|f'(x_0)^{-1}\right\|\Big[\big\| f(x_0)+ f'(x_0)(y_0 -x_0) \big\| + \big\|f({\bar x}) - f(x_0) - f'(x_0)( {\bar x}-x_0)\big\| \Big].
$$
Using \eqref{eq:semsnew}  with $k=0$ and $V_{x_0} = f'(x_0)$,  and the assumptions \eqref{eq:fsecA1} and \eqref{eq:ssecA1} with $x=x_0$, we obtain
\begin{equation}\label{eq:lim1}
\|y_0 - {\bar x}\| \leq \frac{\|f'({\bar x})^{-1}\|}{1 - L \|f'({\bar x})^{-1}\|\|x_0-{\bar x}\|^{\mu}}\left[\eta_0 \|f(x_0)\| + \frac{\mu L}{1 + \mu} \|x_0 - {\bar x}\|^{1+\mu} \right].
\end{equation}
Owing to $f({\bar x}) = 0$,  from  \eqref{eq:gamma1} we conclude that  $\|f(x_0)\|  \leq \Gamma\|x_0 - {\bar x}\|$. Hence, \eqref{eq:lim1}  becomes
\begin{equation}\label{eq:lim2}
\|y_0 - {\bar x}\| \leq  \frac{\|f'({\bar x})^{-1}\|[\eta_0 \Gamma (1+\mu) + \mu L \|x_0 - {\bar x}\|^{\mu}]}{(1 + \mu)[1 - L\|f'({\bar x})^{-1}\|\|x_0 - {\bar x}\|^{\mu}]} \|x_0 - {\bar x}\|.
\end{equation}
On the other hand, letting   $x_1 \in P_C(y_0, x_0, \theta_0)$,  Lemma~\ref{pr:condi} with $y = y_0$, $x=x_0$, $\theta = \theta_0$, $\tilde y={\bar x}$, and $\tilde x={\bar x}$,  implies that
$$
\left\|x_1 -{\bar x} \right\| \leq  \left\|y_0-{\bar x}\right\| + \sqrt{2\theta_0} \left\|y_0 - x_0\right\|\leq  \left\|y_0-{\bar x}\right\|  \left(1 + \sqrt{2\theta_0} \right) + \sqrt{2\theta_0} \left\|x_0-{\bar x} \right\|.
$$
Thus, combining the inequality \eqref{eq:lim2} with  the last inequality, we conclude that
$$
\left\|x_1-{\bar x} \right\|\leq  \frac{\|f'({\bar x})^{-1}\|[\eta_0 \Gamma (1+\mu) + \mu L \|x_0 - {\bar x}\|^{\mu}]}{(1 + \mu)[1 - L\|f'({\bar x})^{-1}\|\|x_0 - {\bar x}\|^{\mu}]}\left\|x_0-{\bar x} \right\| \left(1 + \sqrt{2\theta_0}\right) + \sqrt{2\theta_0} \left\|x_0-{\bar x} \right\|,
$$
which it is equivalent to \eqref{eq:BoundIneT2} for $k=0$. As $\eta_0 < \eta$ and $\theta_0 < \theta$,  by using \eqref{eq:EtaT2} and  \eqref{eq:DeltaT2}, we have
\begin{multline*}
\frac{\|f'({\bar x})^{-1}\|[\eta_0 \Gamma (1+\mu) + \mu L \|x_0 - {\bar x}\|^{\mu}]}{(1 + \mu)[1 - L\|f'({\bar x})^{-1}\|\|x_0 - {\bar x}\|^{\mu}]}\left(1+\sqrt{2\theta_0}\right) +  \sqrt{2\theta_0} < \\
\frac{\|f'({\bar x})^{-1}\|[\eta \Gamma (1+\mu) + \mu L { \delta}^{\mu}]}{(1 + \mu)[1 - L\|f'({\bar x})^{-1}\|{ \delta}^{\mu}]}\left(1+\sqrt{2\theta}\right) +  \sqrt{2\theta} < 1.
\end{multline*}
Then, because $x_0\in B_{\delta}({\bar x})$  we obtain from \eqref{eq:BoundIneT2} with $k=0$, that  $\left\|x_1 -{\bar x}  \right\| <  \|x_0 -{\bar x} \|<\delta$. As  $P_C(y_0, x_0, \theta_0)$ belongs to $C$ and $ x_1\in P_C\left(y_0, x_0, \theta_0\right)$,  we conclude that $ x_1$   belongs to  $B_\delta({\bar x})\cap C$, which completes the induction step for $k=0$.   The general  induction step is completely analogous.  Therefore,  every  sequence $\{x_k\}$ generated by Algorithm~\ref{Alg:INP} is contained in $B_\delta({\bar x})\cap C$ and  satisfies  \eqref{eq:BoundIneT2}. Now, we proceed to prove that the sequence $\{x_k\}$ converges to ${\bar x}$. As $0\leq\theta_k<{\theta}$ and $0 \leq\eta_k < {\eta}$ for all $k=0, 1, \ldots$, it follows from \eqref{eq:EtaT2}, \eqref{eq:BoundIneT2}, and \eqref{eq:DeltaT2} that
$$
\|x_{k+1}- {\bar x}\| <  \left[\frac{\|f'({\bar x})^{-1}\|[\eta \Gamma(1 + \mu) + \mu L\|x_k-{\bar x} \|^{\mu}] }{(1+\mu)[1-  L\|f'({\bar x})^{-1}\|\|x_k-{\bar x} \|^\mu]}\left(1+\sqrt{2\theta}\right) +  \sqrt{2\theta}\right]\|x_k - {\bar x}\|< \|x_k - {\bar x}\|,
$$
for all $k = 0, 1, \ldots$. This implies that the sequence $\{\|x_k - {\bar x} \|\}$ converges.  Let us say that ${\bar t}:=\lim_{k\to +\infty}\|x_k - {\bar x}\| \leq \delta$. Thus, taking the  limit in the last inequality   as  $k$ goes to $+ \infty$,  we obtain
$$
{\bar t} \leq \left[\frac{\|f'({\bar x})^{-1}\|[\eta \Gamma(1 + \mu) + \mu L{\bar t}^{\mu}] }{(1+\mu)[1- L\|f'({\bar x})^{-1}\|{\bar t}^\mu]}\left(1+\sqrt{2\theta}\right) +  \sqrt{2\theta}\right]{\bar t},
$$
If ${\bar t}\neq 0$, then \eqref{eq:EtaT2} and  \eqref{eq:DeltaT2}  implies that  $ {\bar t} < {\bar t}$, which is absurd.  Hence,  ${\bar t} = 0$ and consequently  $\{x_k\}$ converge $Q$-linearly to ${\bar x}$.  Assuming  that  $\lim_{k \to +\infty} \theta_k = 0$ and $\lim_{k \to +\infty} \eta_k = 0$,     it  follows  from \eqref{eq:BoundIneT2}  that
$$
\lim_{k\to +\infty} \frac{\|x_{k+1}- {\bar x}\|}{\|x_{k}- {\bar x}\|}=0.
$$
Hence, $\{x_k\}$ converges $Q$-superlinearly    to ${\bar x}$.  Now, we assume that  $\eta_k <  \min \{ {\eta}\|f(x_k)\|^{\mu}, \eta\}$ and $\theta_k < \min \{{\theta}\|f(x_k)\|^{2\mu}, \theta\}$.   Considering that  $\{x_k\}$   belongs to  $ B_\delta({\bar x})$, $f({\bar x}) = 0$, and $\delta < r$, it follows from  \eqref{eq:gamma1} that $\|f(x_k)\|  \leq \Gamma\|x_k - {\bar x}\|$ for all $ k = 0, 1, \ldots$. Thus, $\eta_k < \eta \Gamma^{\mu}\|x_k - {\bar x}\|^{\mu}$ and  $\theta_k < \theta \Gamma^{2\mu}\|x_k - {\bar x}\|^{2\mu}$ for all $ k = 0, 1, \ldots$. Then, \eqref{eq:BoundIneT2} implies
\begin{multline*}
\|x_{k+1} - {\bar x}\|  <  \left[\frac{ \|f'({\bar x})^{-1}\|[\eta \Gamma^{1+\mu}(1 + \mu) + \mu L] }{(1+\mu)[1- L{ \|f'({\bar x})^{-1}\|}\|x_k-{\bar x} \|^\mu]}\left(1+\Gamma^{\mu}\sqrt{2\theta} \|x_k - {\bar x}\|^{\mu}\right)\right. +  \\ \left.\Gamma^{\mu}\sqrt{2\theta}\right] \|x_k-{\bar x} \|^{1+\mu},
\end{multline*}
for all $k = 0, 1, \ldots$. Therefore, $\{x_k\}$ converges to ${\bar x}$ with  order $1 + \mu$, which complete  the proof.
\end{proof}

In the next remark, we present a particular case of Theorem~\ref{th:mainInex}, where  the projection and  Newton method are exact.
\begin{remark}
In Theorem~\ref{th:mainInex}, if we take $\theta_k\equiv 0$ and $\eta_k \equiv 0$, then  for $0< \mu \leq 1$,  the convergence  rate  is $1 + \mu $ as follows
$$
\| x_{k+1} - {\bar x}\| \leq \frac{\mu L\|f'({\bar x})^{-1}\| }{(1+\mu)[1-   L\|f'({\bar x})^{-1}\|\|x_k-{\bar x} \|^\mu]} \|x_k - {\bar x}\|^{1 + \mu}, \qquad k = 0, 1, \ldots.
$$
\end{remark}
\section{Special cases}\label{sec:specas}
In this section, we present  two special cases: one  of  Theorem~\ref{th:conv} and one of  Theorem~\ref{th:mainInex}.  We begin by presenting the special case of  Theorem~\ref{th:conv}.
\subsection{Under semismooth condition}
In this section, we present a local convergence theorem for the inexact Newton-InexP method  for solving constrained semismooth equations. The semismoothness plays an important role,  since the Newton method is still applicable and converges locally with superlinear rate to a regular solution. Let us first to present the concept of regularity.

\begin{definition}
Let $\Omega \subseteq \mathbb{R}^n$ be an open set. A locally Lipschitz continuous function  $f: \Omega  \to \mathbb{R}^n$ is said to be  regular at ${\bar x}\in \Omega $ if every $V_{\bar x}\in \partial f({\bar x})$ is nonsingular. If $f$ is regular at all points of $\Omega$,  the function $f$  is said to be regular on $\Omega$.
\end{definition}

In the following, our first task is to prove that locally Lipschitz continuous functions satisfy \eqref{eq:fcA1} near a regular point for every  $0<\epsilon<1/ \lambda_{\bar x}$.  First,  we remind that $\partial f({x})$ is a nonempty and compact set for all $x\in  \Omega$. See \cite[Proposition~2.6.2, p. 70]{Clarke1990}. The statement  of the result is as follows.
\begin{lemma}\label{lem:semism}
Let $\Omega \subseteq \mathbb{R}^n$ be an open set and  $f: \Omega \to \mathbb{R}^n$ be a locally Lipschitz continuous function. If $f$ is regular at  $ {\bar x}\in \Omega$, then  for every $0<\epsilon < 1/\lambda_{\bar x}$, where $\lambda_{\bar x}\geq \max\{\|{V_{\bar x}^{-1}}\|: ~ { V_{\bar x}}\in \partial f({\bar x})\}$,  there exists $\delta >0$ such that $f$ is regular on  $B_{\delta}({\bar x})$ and there holds
\begin{equation} \label{eq:BanachLem}
\|V_{x}^{-1}\| \leq  \frac{{\lambda_{\bar x}}}{1 - {\lambda_{\bar x}}\epsilon}, \qquad \forall ~V_x\in \partial f(x), \quad \forall~x\in  B_{\delta}({\bar x}).
\end{equation}
\end{lemma}
\begin{proof}
As $f$ is regular at ${\bar x \in \Omega}$ and $\partial f({\bar x})$ is a nonempty and compact set, $\lambda_{\bar x} >0$ is well-defined. On the other hand, it follows from  \cite[Proposition 2.6.2, p.~70]{Clarke1990}  that $\partial f$ is upper semicontinuous at ${\bar x}$. Thus, for every $\epsilon >0$  there exists $\delta>0$ such that
$$
\partial f(x) \subset \left\{V_x \in \mathbb{R}^{n\times n}: ~ \|V_x - V_{\bar x}\|< \epsilon, ~  \mbox{for some}~   V_{\bar x} \in \partial f({\bar x})\right\},  ~  \forall~x\in  B_{\delta}({\bar x}).
$$
Hence, for each   $V_x\in \partial f(x)$ and  $0<\epsilon < 1/\lambda_{\bar x}$,  there exists $V_{\bar x} \in \partial f({\bar x})$ that is nonsingular such that    $\| V_{\bar x}^{-1}\|\|V_x - V_{\bar x}\|<  \lambda_{\bar x} \epsilon< 1$. Thus, applying the Banach lemma, see \cite[Lemma 5A.4, p.~282]{DontchevRockafellar2010Book},  we conclude  that $V_x$ is nonsingular and
$$
\|V_{x}^{-1}\| \leq \frac{\| V_{\bar x}^{-1}\|}{1 - \| V_{\bar x}^{-1}\|\|V_x - V_{\bar x}\|}.
$$
Therefore, considering that $\| V_{\bar x}^{-1}\| \leq \lambda_{\bar x}$, the inequality \eqref{eq:BanachLem} follows, and the proof is complete.
\end{proof}

In the following,  we present a class of functions satisfying the inequality \eqref{eq:scA1}, namely the  semismooth functions. There are several equivalent definitions for  semismooth functions,  here we  use  that given in \cite[p.~411]{DontchevRockafellar2010Book}.  For an extensive study on  semismooth  functions,  see, for example,   \cite{Facchinei2003}.
\begin{definition} \label{Def:DefSS}
Let $\Omega\subseteq \mathbb{R}^n$ be an open set.  A function $f: \Omega \to \mathbb{R}^n$ that is locally Lipschitz continuous   on $\Omega$  and directionally differentiable in every direction is said to be  semismooth at ${\bar x} \in \Omega$  when for every  $\epsilon>0$   there exists $\delta>0$ such that
\begin{equation*} \label{eq:CalLem}
\|f({\bar x})-f(x)-V_x({\bar x}-x)\|\leq \epsilon \|x-{\bar x}\|,\qquad \forall~V_x\in \partial f(x), \quad\forall~x\in B_{\delta}({\bar x}),
\end{equation*}
and is said to be  $\mu$-order semismooth  at ${\bar x} \in \Omega$, for $0<\mu\leq1$   when there exist $\epsilon >0$ and  $\delta>0$ such that
\begin{equation*} \label{eq:MuSs}
\|f({\bar x})-f(x)-V_x({\bar x}-x)\|\leq \epsilon \|x-{\bar x}\|^{1+\mu},\qquad \forall~V_x\in \partial f(x), \quad \forall~x\in B_{\delta}({\bar x}).
\end{equation*}
\end{definition}

Next, we state and prove the local convergence result of the inexact Newton-InexP method for solving constrained semismooth equations.
\begin{theorem}\label{th:mainss}
Let $\Omega \subseteq {\mathbb R}^n$ be an open set, $C\subset \Omega$ be a closed convex set, and  $f: \Omega \to {\mathbb R}^n$ be semismooth and regular at ${\bar x} \in \Omega$. Let     $\Gamma>0$ and $0<r \leq {\bar r}:=\sup\left\{t\in {\mathbb R}:~ B_t({\bar x})\subset \Omega\right\}$ such that
$$
\|f(x) - f({\bar x})\| \leq \Gamma \|x - {\bar x}\|, \qquad \forall~x\in  B_{ r}({\bar x}).
$$
Take    ${\theta} > 0$ and    ${\eta}> 0$  such that
$$
{\theta}< \frac{1}{2}, \qquad \qquad {\eta}<\frac{1-\sqrt{2{\theta}}}{\lambda_{\bar x}\Gamma\left(1+\sqrt{2{\theta}}\right)}.
$$
Assume that  ${\bar x}\in C$ and   $f({\bar x}) = 0$.    Then,   there exists   $0<{ \delta}\leq r$  such that  every sequence $\{x_k\}$  generated by Algorithm~\ref{Alg:INP}  starting in $x_0 \in C \cap B_\delta({\bar x})\backslash \{{\bar x}\}$,  with  $0\leq \theta_k < {\theta}$ and  $0\leq \eta_k <{\eta}$ for all $k=0, 1, \ldots$, belongs to  $ B_\delta({\bar x})\cap C$ and converges $Q$-linearly to ${\bar x}$.  If  $\lim_{k \to +\infty} \theta_k =0$ and $ \lim_{k \to +\infty} \eta_k = 0$, then $\{x_k\}$ converges $Q$-superlinearly to ${\bar x}$. In addition, if  $f$ is $\mu$-order semismooth at ${\bar x}$,  $\eta_k <  \min \{ {\eta}\|f(x_k)\|^{\mu}, \eta\}$, and $\theta_k < \min \{{\theta}\|f(x_k)\|^{2\mu}, \theta\}$, then the convergence of $\{x_k\}$ to ${\bar x}$ is  of the order of $1 + \mu$.
\end{theorem}
\begin{proof}
As the function $f$ is  semismooth  and regular at ${\bar x} \in \Omega$, we can take $\lambda_{\bar x}\geq \max\{\|{V_{\bar x}^{-1}}\|: ~ { V_{\bar x}}\in \partial f({\bar x})\}$.   Take $0< \epsilon<1/\lambda_{\bar x}$. Then, from Lemma~\ref{lem:semism}  and Definition~\ref{Def:DefSS}, there exists $0<\delta\leq \min\{ r, 1\}$ satisfying \eqref{eq:fcA1} and \eqref{eq:scA1} for $\mu=0$. In addition,  if  $f$ is  $\mu$-order semismooth, we conclude also  from  Lemma~\ref{lem:semism}  and Definition~\ref{Def:DefSS} that  there exists $0<\delta\leq \min\{ r, 1\}$ satisfying \eqref{eq:fcA1} and \eqref{eq:scA1} for $0<\mu\leq 1$.  Therefore,    $f$ satisfies all conditions of Theorem~\ref{th:conv} and  by reducing   $\epsilon> 0$ so that it satisfies the second  inequality in  \eqref{eq:DeltaEtaT1}, the result follows.
\end{proof}

In the following  remark, we show that with some adjustments Theorem~\ref{th:mainss}  reduces to some well-known results.
\begin{remark}
It is worth mentioning that if $C = {\mathbb R}^n$ and $\theta_k = 0$ for all $k = 0,1,\ldots$, then  with some adjustments Theorem~\ref{th:mainss}  reduces to \cite[Theorem 3]{JoseMariLiqun1995}; see also  \cite[Theorem~7.5.5, p. 694]{Facchinei2003}. If $C = {\mathbb R}^n$,  $\eta_k = \theta_k = 0$ for all $k = 0,1,\ldots$, then Theorem~\ref{th:mainss} reduces to \cite[Theorem~3.2]{Qi1993}, see also  \cite[Theorem~7.5.3, p.~693]{Facchinei2003}. Finally, if $C = {\mathbb R}^n$,  $f$ is a continuously differentiable function, $f'({\bar x})$ is nonsingular, and $\theta_k = \eta_k = 0$ for all $k =0,1,\ldots$, then the theorem above reduces to the first part of  \cite[Proposition~1.4.1, p.~90]{Bertsekas1999}.
\end{remark}
\subsection{Under radial H\"older condition on the derivative}
In this section, we present a local convergence theorem for the inexact Newton-InexP method under the radial H\"older condition on the derivative. We begin by presenting the definition of the radial  H\"older condition.
\begin{definition}
Let $\Omega \subseteq \mathbb{R}^n$ be an open set and $f: \Omega  \to \mathbb{R}^n$  be a  continuously differentiable function. The derivative   $f'$   satisfies the radial  H\"older condition  at  ${\bar x}\in \Omega$ if there  exist   $L > 0$ and $0<\mu \leq 1$ such that
$$
\|f'(x) - f'({\bar x} + \tau(x - {\bar x}))\| \leq L(1 - \tau^{\mu})\|x - {\bar x}\|^{\mu},
$$
for all $x\in \Omega$  and  $\tau \in [0,1]$ such that ${\bar x} + \tau(x - {\bar x})\in \Omega$.
\end{definition}

Our first task is to prove that continuously differentiable  functions with radially H\"older derivative  satisfy \eqref{eq:fsecA1} around regular points.
\begin{lemma}\label{lem:Holder1}
Let $\Omega \subseteq {\mathbb R}^n$ be an open set and $f: \Omega \to \mathbb{R}^n$ be a continuously differentiable  function.  Assume that $f'$ is nonsingular and    radially  H\"older at  ${\bar x}\in \Omega$,  with constants  $L > 0$  and $0< {\mu} \leq 1$.  Take
\begin{equation} \label{eq:delta6}
0<{\hat r}< \frac{1}{(L{\|f'(\bar x})^{-1}\|)^{1/{\mu}}}.
\end{equation}
Then, $f'(x)$ is nonsingular for all $x \in B_{\hat r}({\bar x})$,  and the following holds:
\begin{equation*} \label{eq:deltabl}
\left\|f'(x)^{-1}\right\|\leq \frac{{\|f'(\bar x})^{-1}\|}{1 -L{\|f'(\bar x})^{-1}\|\|x - {\bar x}\|^{\mu}},  \qquad \forall~ x \in B_{\hat r}({\bar x}).
\end{equation*}
\end{lemma}
\begin{proof}
Using that    $f'$ is nonsingular and    radially  H\"older at  ${\bar x}\in \Omega$,  with constants  $L > 0$  and $0< {\mu} \leq 1$, and taking into account \eqref{eq:delta6}, we have
$$
\|f'({\bar x})^{-1}\|\|f'(x) - f'({\bar x})\| \leq  L \|f'({\bar x})^{-1}\|  \|x - {\bar x}\|^{\mu}< L \|f'({\bar x})^{-1}\| {\hat r}^{\mu} <1,
$$
for all $x \in B_{\hat r}({\bar x})$. Thus, the lemma  follows by applying the Banach lemma  \cite[Lemma 5A.4, p.~282]{DontchevRockafellar2010Book}.
\end{proof}

The next lemma establishes  that continuously differentiable functions with   radially  H\"older derivative  satisfy \eqref{eq:ssecA1}; its proof follows the same idea as \cite[Proposition~1.4.1, p.~90]{Bertsekas1999} and will be omitted here.
\begin{lemma}\label{lem:Holder2}
Let  $\Omega \subseteq {\mathbb R}^n$ be an open set,  ${\bar x}\in \Omega$, ${\bar r}:= \sup\{t \in \mathbb{R}:~ B_t({\bar x})\subset \Omega\}$, and $f: \Omega \to \mathbb{R}^n$ be a continuously differentiable function. Assume that $f'$ is   radially  H\"older at  ${\bar x}$,  with constants  $L > 0$  and $0<{\mu} \leq 1$.  Then it holds that
\begin{equation*}\label{eq:Holder}
\left\| f({\bar x}) - f(x)- f'(x)( {\bar x}-x)\right\|\leq \frac{{\mu} L}{1+{\mu}}\|x - {\bar x}\|^{1+ {\mu}},  \qquad \forall~ x \in B_{\bar r}({\bar x}).
\end{equation*}
\end{lemma}

Now, we are ready to present a local convergence theorem on the inexact Newton-InexP method for a continuously differentiable  function $f$ such  that $f'$ is radially  H\"older. The statement of the result is as follows.
\begin{theorem}\label{th:mainHolder}
Let $\Omega \subseteq {\mathbb R}^n$ be an open set, $C\subset \Omega$ be a closed convex set, and $f: \Omega \to \mathbb{R}^n$  be a continuously differentiable  function  such that  $f'$ is nonsingular and    radially  H\"older at  ${\bar x}\in \Omega$, with constants  $L > 0$  and $0<{\mu} \leq 1$.   Let $\Gamma>0$ and $0<r \leq {\bar r}:=\sup\left\{t\in {\mathbb R}:~ B_t({\bar x})\subset \Omega\right\}$ such that
$$
\|f(x) - f({\bar x})\| \leq \Gamma \|x - {\bar x}\|, \qquad \forall~x\in  B_{ r}({\bar x}).
$$
Let    ${\theta}> 0$ and    ${\eta}> 0$  such that
$$
{\theta}< \frac{1}{2}, \qquad \qquad {\eta}<\frac{1-\sqrt{2{\theta}}}{\Gamma\|f'({\bar x})\|\left(1+\sqrt{2{\theta}}\right)}.
$$
Assume that  ${\bar x}\in C$ and   $f({\bar x}) = 0$.    Then,   there exists   $0<\delta\leq r$  such that  every sequence $\{x_k\}$  generated by Algorithm~\ref{Alg:INP}  starting in $x_0 \in C \cap B_\delta({\bar x})\backslash \{{\bar x}\}$,  with  $0\leq \theta_k< {\theta}$ and  $0\leq \eta_k < {\eta}$ for all $k=0, 1, \ldots$, belongs to $ B_\delta({\bar x})\cap C$ and converges $Q$-linearly to ${\bar x}$.  As a consequence, if  $\lim_{k \to +\infty} \theta_k =0$ and $ \lim_{k \to +\infty} \eta_k = 0$, then $\{x_k\}$ converges $Q$-superlinearly to ${\bar x}$.  In addition, if  $\eta_k <  \min \{ {\eta}\|f(x_k)\|^{\mu}, \eta\}$  and $\theta_k < \min \{{\theta}\|f(x_k)\|^{2\mu}, \theta\}$, then the convergence of $\{x_k\}$ to ${\bar x}$ is  of the order of $1 + \mu$.
\end{theorem}
\begin{proof}
Let $0<{\hat r}<1/[(L{\|f'(\bar x})^{-1}\|)^{1/{\mu}}]$ and  $0<{\hat \delta}\leq\min  \{ {\hat r}, {r}\}$.  Then, from Lemmas~\ref{lem:Holder1} and~\ref{lem:Holder2},  we conclude that $f$ satisfies \eqref{eq:fsecA1} and \eqref{eq:ssecA1} in $B_{\hat \delta}({\bar x})$.   Therefore,    $f$ satisfies all conditions of Theorem~\ref{th:mainInex} and by taking  $\delta> 0$ satisfying  \eqref{eq:DeltaT2}, the result follows.
\end{proof}

In the following  remark, we show that with some adjustments, Theorem~\ref{th:mainHolder}  has as particular instances some well-known results.
\begin{remark}
It is worth mentioning that  if $C = {\mathbb R}^n$ and  $\eta_k = \theta_k = 0$ for all $k = 0,1,\ldots$, then Theorem~\ref{th:mainHolder}  reduces to the second part of  \cite[Proposition~1.4.1, p.~90]{Bertsekas1999}.  If the  procedure to obtain the feasible inexact projection is the CondG procedure and  $\eta_k = 0$ for all $k = 0,1,\ldots$, then  Theorem~\ref{th:mainHolder} reduces to \cite[Theorem~7]{MaxJefferson2017}.  Finally, if the  procedure to obtain the feasible inexact projection is the CondG procedure, then with some adjustments  Theorem~\ref{th:mainHolder} reduces to \cite[Corollary 2]{GoncalvesOliveira2017}.
\end{remark}

\section{Numerical experiments} \label{NunEx}
In this section, we report some numerical experiments that show the computational feasibility of the inexact Newton method with
feasible exact projections (\textit{inexact Newton-ExP method}) and  inexact Newton method with feasible inexact projections (\textit{inexact Newton-InexP method}) on one class  CAVEs.  It is worth mentioning that works  dealing with the Newton method to solve  absolute value equation (AVE) include \cite{BelloCruz2016,Mangasarian2009}. The CAVE is described as
$$
\mbox{find} \quad x \in C \quad \mbox{such that}\quad Ax - |x| = b,
$$
where $C := \{x \in \mathbb{R}^n:~ \sum_{i = 1}^{n}x_i \leq d, \, x_i \geq 0, \, i = 1,2, \ldots,n\}$, $A \in \mathbb{R}^{n\times n}$, $b \in \mathbb{R}^n \equiv \mathbb{R}^{n\times 1}$, and $|x|$ denotes the vectors whose $i$-th component is equal to $|x_i|$.

In our implementation, the CAVEs have been generated randomly. We used the Matlab routine \textit{sprand} to construct matrix $A$. In particular, this routine generates a sparse matrix with predefined dimension, density, and singular values. Initially, we defined the dimension $n$ and randomly generated the vector of singular values from a uniform distribution on $(0, 1)$. To ensure that $\|A^{-1}\|< 1/3$, i.e., so that the assumptions of \cite[Theorem 2]{BelloCruz2016} are fulfilled, we rescale the vector of singular values by multiplying it by $3$ divided by the minimum singular value multiplied by a random number in the interval $(0, 1)$. To generate the vector $b$ and the constant $d$, we chose a random solution $x_*$ from a uniform distribution on $(0.1, 300)$ and computed $b = Ax_* - |x_*|$ and $d = \sum_{i = 1}^{n}(x_*)_i$, where $(x_*)_i$ denotes the $i$-th component of the vector $x_*$. In both methods, $x_0 = (d/2n, d/2n,\ldots, d/2n)$ was defined as the starting point, the initialization data $\theta$ was taken equal to $10^{-1}$ and $10^{-8}$ for the methods with inexact and exact projection, respectively, and $\eta$ was taken equal to $0.9999[(1-\sqrt{2\theta})/0.5\Gamma(1+\sqrt{2\theta})]$ with  $\Gamma = \|A\| + 1$. We stopped the execution of Algorithm~\ref{Alg:INP} at $x_k$, declaring convergence if
$
\|Ax_k - |x_k| - b\| < 10^{-6}.
$
In case this stopping criterion was not respected, the method stopped when a maximum of $50$ iterations had been performed.  The procedure to obtain feasible projections used in our implementation was the \textit{CondG Procedure}; see, for example, \cite{GoncalvesOliveira2017}. In particular, this procedure stopped when either the stopping criterion, i.e., the condition $\langle y_k - x_{k+1}, z - x_{k+1}\rangle \leq  \theta_k\|y_k - x_k\|^2$ was satisfied for all $z \in C$ and $k = 0, 1, \ldots$ or a maximum of $100$ iterations was performed. For this class of problems, an element of the Clarke generalized Jacobian at $x$ (see \cite{BelloCruz2016,Mangasarian2009}) is given by
$$
V = A -  \mbox{diag}(\mbox{sgn}(x)), \qquad x \in  \mathbb{R}^n,
$$
where $\mbox{diag}(\alpha_i)$ denotes a diagonal matrix with diagonal elements $\alpha_1, \alpha_2, \ldots, \alpha_n$ and $\mbox{sgn}(x)$ denotes a vector with components equal to $-1$, $0$, or $1$ depending on whether the corresponding component of the vector $x$ is negative, zero, or positive. The inexact Newton-ExP and inexact Newton-InexP  methods requires the linear system $f(x_k)+ V_k(y_k - x_k) = 0$ to be solved approximately, in the sense of \eqref{eq:semsnew}. Matlab has several iterative methods for solving linear equations. For our class of problems, the routine \textit{lsqr} was the most efficient; thus, in all tests, we used \textit{lsqr} as an iterative method to solve linear equations approximately. In particular, this routine is an algorithm for sparse linear equations and sparse least squares; for further details, see, for example, \cite{Paige1982}. We compare the efficiency and robustness of the methods using the performance profiles graphics, see \cite{Dolan2002}. The efficiency is related to the percentage of problems for which the method was the fastest, whereas robustness is related to the percentage of problems for which the method found a solution. In a performance profile, efficiency and robustness can be accessed on the extreme left (at 1 in domain) and right of the graphic, respectively. The numerical results were obtained using Matlab version R2016a on a 2.5~GHz Intel\textregistered\ Core\texttrademark\ i5 2450M computer with 6~GB of RAM and Windows 7 ultimate system and are freely available from https://orizon.mat.ufg.br/admin/pages/11432-codes.

Figure~\ref{fig1} reports a comparison, using performance profiles, between the inexact Newton-ExP and inexact Newton-InexP  methods for solving CAVEs of dimensions $1000$, $5000$, $8000$, and $10000$. We generated $200$ CAVEs with dimensions $1000$ and $5000$, and $100$ CAVEs with dimensions $8000$ and $10000$. The density of the matrix $A$ was taken equal to $0.003$, as well as in \cite{BelloCruz2016}. This means that only about $0.3\%$ of the elements of $A$ are nonnull. To obtain the CPU time more accurately, we run each test problem 10 times and we defined the corresponding CPU time as the median of these measurements. Analyzing Figure~\ref{fig1}, we see that the inexact Newton-InexP method is more efficient than the inexact Newton-ExP method on the set of test problems. In particular, the efficiencies of the inexact Newton method with the exact and inexact projections are, respectively, $30.5\%$ and $69.5\%$ for problems of dimension $1000$, $31.0\%$ and $69.0\%$ for problems of dimension $5000$, $41.0\%$ and $59.0\%$ for problems of dimension $8000$, and $30.0\%$ and $70.0\%$ for problems of dimension $10000$. Thus, we can conclude that for this class of test problem the parameter $\theta$ and consequently $\eta$ given in \eqref{eq:DeltaEtaT1} limit the effectiveness of the method.

\begin{table}[h]\caption{Performance of the inexact Newton-ExP method versus the inexact Newton-InexP method}\label{tab1}
{\tiny
\resizebox{\textwidth}{!}{
\begin{tabular}{lllllllllll}
\hline\noalign{\smallskip}
  &   \multicolumn{3}{l}{Inexact Newton-ExP method} & \multicolumn{3}{l}{Inexact Newton-InexP method} \\
\noalign{\smallskip}\hline\noalign{\smallskip}
Dimension   & $\%$ & Iter & Time & $\%$ & Iter &Time \\
\noalign{\smallskip}\hline
\\
1000  & 100.0 & $6.61$ & $0.58$ & 100.0 & $5.50$ & $0.56$
\\
5000  & 100.0 & $6.70$ & $11.67$ & 100.0 & $5.67$ & $11.48$
\\
8000  & 100.0 & $6.90$ & 30.76 & 100.0 & $5.81$ & $30.42$
\\
10000  & 100.0& $6.88$ & $46.66$ & 100.0 & $5.77$ & $45.27$
\\
\noalign{\smallskip}\hline
\end{tabular}}}
\end{table}

{\footnotesize
\begin{figure}[h!]
\centering
\subfloat[$n = 1000$]{
\includegraphics[width=0.514\textwidth]{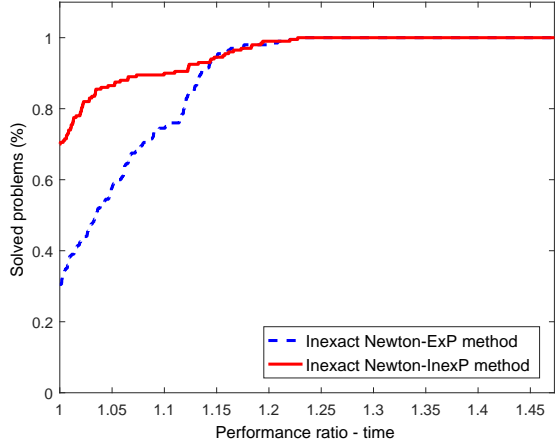}}
\subfloat[$n = 5000$]{
\includegraphics[width=0.514\textwidth]{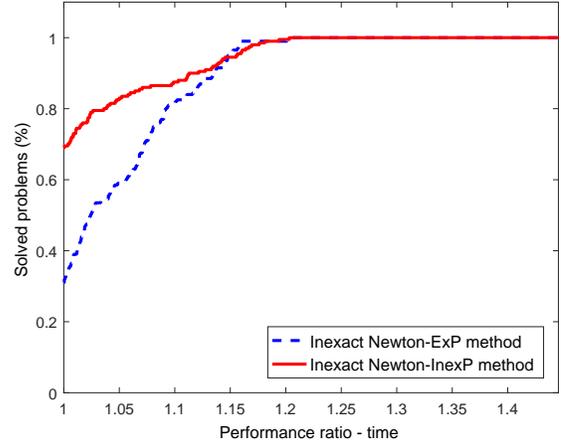}}
\\
\subfloat[$n = 8000$]{
\includegraphics[width=0.514\textwidth]{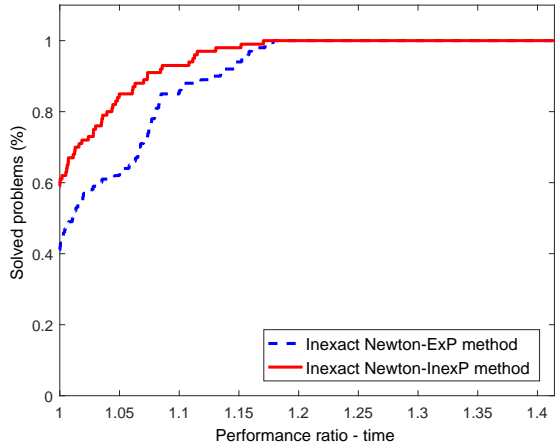}}
\subfloat[$n = 10000$]{
\includegraphics[width=0.514\textwidth]{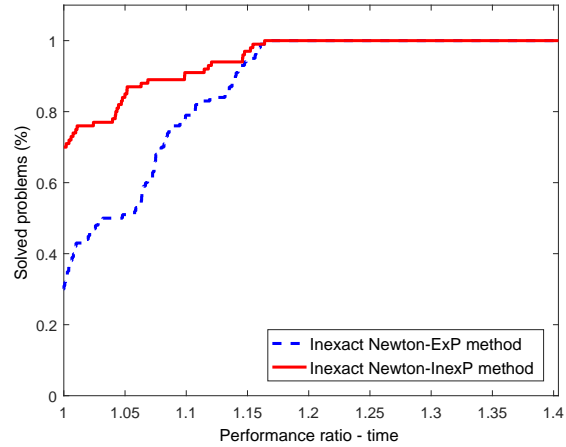}}
\caption{Performance profile comparing the inexact Newton-ExP method versus the inexact Newton-InexP method for CAVEs using CPU time as the performance measurement.}
\label{fig1}
\end{figure}
}

Table~\ref{tab1} lists, for each method, the percentage of problems solved ``$\%$'', the average numbers of iterations ``Iter'', and the average times in seconds ``Time''. As can be seen, the robustness is $100.0\%$ for both methods. The average numbers of iterations is approximately $7$ and $6$ for the exact and inexact versions, respectively. Moreover, with respect to the average time, it is possible to observe a certain trend, that is, as the dimension of the problem increases, the performance of the inexact Newton-InexP method becomes better compared with the inexact Newton-ExP method.

The results discussed above allow us to conclude that there are problems for which the use of the inexact projection can become the more efficient method. Thus, we can say that the inexact Newton-InexP method may be a robust and efficient tool for solving other classes of nonsmooth functions subject to a set of constraints.

\section{Conclusions} \label{Sec:Conclusions}
We know that to solve nonlinear equations, the Newton method is the starting point for designing many more sophisticated  methods, including the Gauss--Newton method,  Levenberg--Marquardt method, trust region method, and several other variants;  see \cite{DennisSchnabel1983} for a comprehensive study on this subject.  Therefore, the study of new properties of this method is important in itself.  In this paper, we have proposed a  new scheme for solving constrained smooth and nonsmooth equations, the essence of which was to combine the exact/inexact Newton method with a feasible inexact projection. In Theorems~\ref{th:mainss} and \ref{th:mainHolder} we have shown that,  under mild assumptions, the exact/inexact Newton-InexP method for solving constrained smooth and nonsmooth equations preserves the local convergence properties if  feasible inexact projections with  suitable error relative tolerance are used. In particular, under the standard nonsingularity condition, the superlinear/Q-quadratic rate is preserved. In this sense, we expect that our results  become a first step towards a study of the behavior  of the Newton method  and  aforementioned variants,  with feasible inexact projections,  under more reasonable regularity conditions. It is worth mentioning that  inexact projection, as proposed in \cite{BehlingFischerHerrichIusemYe2014}, is in general infeasible and  cheaper than the one  proposed in our paper. However,  as far as we know,  even under nonsingularity conditions, there are no results showing that the Newton method or any of its  variants for solving constrained nonsmooth equations  allowing infeasible inexact projections maintains  its convergence local properties, namely superlinear and/or Q-quadratic rate.  We think this question  should be investigated, though we believe that  infeasible  inexact projections should limit the convergence rate of the method as a whole.

Finally, to show the practical behavior of the proposed method, we have tested it on some medium- and large-scale CAVEs. The numerical experiments have shown that the inexact Newton-InexP method works quite well for solving this class of problems.


\end{document}